\documentclass{article}

\usepackage{arxiv}

\usepackage[utf8]{inputenc}
\usepackage[T1]{fontenc}
\usepackage{hyperref}
\usepackage{url}
\usepackage{booktabs}
\usepackage{amsfonts}
\usepackage{nicefrac}
\usepackage{microtype}

\usepackage{latexsym,amssymb,amsmath,tikz,verbatim,cancel,comment}

\numberwithin{equation}{section}
\newtheorem{theorem}{Theorem}[section]
\newtheorem{lemma}[theorem]{Lemma}

\newtheorem{corollary}[theorem]{Corollary}

\newtheorem{definition}[theorem]{Definition} 

\newtheorem{remark}[theorem]{Remark}
\newtheorem{example}[theorem]{Example}

\newenvironment{proof}[1][\textbf{Proof}]{\noindent\textbf{#1. }}{\ $\square$}

\title{Powers of Principal $Q$-Borel ideals}

\author{
    Eduardo Camps Moreno\\
  Department of Mathematics\\
  Escuela Superior de F\'isica y Matem\'aticas\\
  Mexico City, Mexico \\
  \texttt{camps@esfm.ipn.mx} \\
  
   \And
 
 Craig Kohne \\
  Department of Mathematics and Statitics\\
  McMaster University\\
  Hamilton, ON, Canada \\
  \texttt{kohnec@math.mcmaster.ca} \\
  
  \And
  
  Eliseo Sarmiento\\
  Department of Mathematics\\
  Escuela Superior de F\'sica y Matem\'aticas\\
  Mexico City, Mexico\\
  \texttt{esarmiento@ipn.mx}\\
  
  \And
  
  Adam Van Tuyl\\
  Department of Mathematics and Statistics\\
  McMaster University\\
  Hamilton, ON, Canada\\
  \texttt{vantuyl@math.mcmaster.ca}
  }

\begin{document}
\maketitle

\begin{abstract}
Fix a poset $Q$ on $\{x_1,\ldots,x_n\}$.
A $Q$-Borel monomial ideal $I \subseteq
\mathbb{K}[x_1,\ldots,x_n]$ is
a monomial ideal whose monomials are closed under the Borel-like moves induced by $Q$. 
A monomial ideal $I$ is a principal
$Q$-Borel ideal, denoted $I=Q(m)$, if there
is a monomial $m$ such that all the minimal
generators of $I$ can be obtained via $Q$-Borel moves from $m$.  In this paper
we study powers of principal $Q$-Borel ideals.
Among our results, we show that all
powers of $Q(m)$ agree with their symbolic
powers, and that the ideal $Q(m)$ satisfies the persistence
property for associated primes.  We also
compute the analytic spread of $Q(m)$
 in terms of the poset $Q$.
 \end{abstract}


\section{Introduction}

Throughout this paper, 
$S = \mathbb{K}[x_1,\ldots,x_n]$ denotes the polynomial
ring over an arbitrary field $\mathbb{K}$.
Francisco, Mermin, and Schweig \cite{qborel}
introduced the notion of a $Q$-Borel
monomial ideal to generalize the properties
of Borel monomial ideals, also called
strongly stable monomial ideals (see
\cite{fms,H} and their references for more
on Borel ideals and their importance). Specifically,
we fix
a poset $Q$ on the set $\{x_1,\ldots,x_n\}$.  Then
a monomial ideal $I$ is a \emph{$Q$-Borel ideal}
if for any monomial $m \in I$, if $x_i|m$ and
$x_j \leq_Q x_i$, then 
$x_j\cdot \frac{m}{x_i} \in I$.
We call $x_j\cdot \frac{m}{x_i}$ a 
\emph{$Q$-Borel move} of $m$.
A \emph{Borel ideal} is then the special instance
when $Q$ is the chain $Q = C: x_1 < x_2 < \cdots < x_n$.
A monomial ideal $I$ is a \emph{principal $Q$-Borel ideal}, denoted $Q(m)$,
if there is a monomial $m$ such that all the
minimal generators of $I$ can be obtained from
$m$ via $Q$-Borel moves.  As shown in 
\cite{qborel} and Bhat's thesis \cite{BPhd},
many properties of $Q(m)$, e.g., projective
dimension, primary decomposition, can be
described in terms of 
the poset $Q$ and order ideals of $Q$ associated
with the monomial $m$.

Our goal in this paper is to study the properties of
powers of principal $Q$-Borel ideals.  Understanding
powers of ideals figures prominently in 
commutative algebra.  Two examples of this theme
are the ideal containment problem and the persistence
of primes.  The ideal containment problem compares
the regular powers of an ideal with its symbolic
powers. The persistence of primes asks whether 
${\rm ass}(I^s) \subseteq {\rm ass}(I^{s+1})$ for 
all $s \geq 1$, where ${\rm ass}(J)$ denotes the 
set of associated primes of $J$.   The references
\cite{BH,spmi,WalSQF,ELS,GGSVT,herzogas,HRV,HH,appei}
form a small subset of papers on these topics; see also \cite{CHHVT,DDGHN} for an introduction.

For principal
$Q$-Borel ideals $Q(m)$, we consider these (and 
other) problems.  Many of our results are expressed in terms
of the combinatorics of the poset of $Q$, thus building
upon \cite[Question 1.3]{qborel} which asked what other
properties of $Q$-Borel ideals are determined by $Q$.  One theme that becomes apparent is that principal
$Q$-Borel ideals satisfy many of the same properties as
principal monomial ideals (in fact, results about principal
monomial ideals become special cases of our work when
$Q$ is the anti-chain).

We first compare the regular and symbolic powers (formal
definitions postponed until later in the paper) of
principal $Q$-Borel ideals.   Our main result in
this direction is:

\begin{theorem}[Theorem \ref{maintheoremsymbolic}]\label{main1}
Let $I = Q(m)$ for some monomial $m$ and
poset $Q$. Then 
$$I^{(d)}=I^d ~~\mbox{for all $d  \geq 1$}.$$
\end{theorem}
\noindent
Our proof requires Francisco, Mermim, and Schweig's \cite{qborel} characterization of the associated primes of 
$Q(m)$, and Cooper, Embree, H\`a, and Hoefel's \cite{spmi} 
description of the symbolic powers of monomial ideals.  
As a corollary, we obtain
results on the Waldschmidt constant, the symbolic defect,
and the resurgence (see Corollary \ref{symboliccor}).

The \emph{analytic spread} of $I$, denoted $\ell(I)$, is the 
Krull dimension of the ring 
\[\mathcal{F}(I) = \bigoplus_{i \geq 0} \frac{I^i}{{\bf m}I^i}
~~~~\mbox{where $I^0 = S$ and ${\bf m } =\langle x_1,\ldots,x_n\rangle$}.\]
For principal $Q$-Borel ideals, we obtain the following
formula for the analytic spread in terms of
combinatorics of $Q$.

\begin{theorem}[Theorem \ref{maintheorem1}]
Let $I=Q(m)$ be a principal $Q$-Borel ideal, let $A(m)$ be the order ideal generated by the support of $m$. Then
    $$\ell(I)=|A(m)|-K(A(m))+1$$
where $K(A(m))$ is the number of connected components in the subposet induced by $A(m)$.
\end{theorem}
\noindent
Our proof uses the fact that for ideals generated by monomials of the same degree, the analytic spread is the rank of the matrix of exponent vectors of the generators.
The analytic spread of $Q(m)$ could
also be computed using results of Herzog,
Rauf, and Vladoiu \cite{HRV}, but our result highlights
the connection to the poset of $Q$.

Herzog, Rauf,  and Vladoiu's paper \cite{HRV} is used
to address the question of persistence of primes.  Precisely, we show that ${\rm ass}(I) = {\rm ass}(I^s)$ for all $s \geq 1$
for any principal $Q$-Borel ideal (see Theorem \ref{persistence}).  In fact, we give two different
proofs for this result.

We also consider powers of 
\emph{square-free principal $Q$-Borel ideals}, denoted
$sfQ(m)$. These square-free monomial ideals are
generated by the square-free monomial generators of
$Q(m)$.   For this class of ideals, we also compute
their analytic spread (see Theorem \ref{sqfanalyticspread}) in terms
of $Q$.

Our paper is structured as follows.
Section 2 is the background on monomial
ideals, posets, and (principal) $Q$-Borel ideals. In section 3 we prove Theorem \ref{maintheoremsymbolic}.  
In Section 4
we examine the persistence of primes problem.
Section 5 is devoted to the analytic spread of
(square-free) principal $Q$-Borel ideals.  

\medskip

\noindent
{\bf Acknowledgments.}  We thank Ashwini Bhat, Chris Francisco, and the two referees for their comments and improvements.
Camps is supported by Conacyt. Sarmiento's research is supported by SNI-Conacyt. Camps and Sarmiento are supported by PIFI IPN 20201016.
 Van Tuyl’s research is supported by NSERC Discovery Grant 2019-05412. 

\section{Background}

In this section we recall the relevant background and definitions.   

\subsection{Basics of monomial ideals and posets}
Given a monomial $m = x_1^{a_1}\cdots x_n^{a_n}$ in $S$, we may
write the monomial as
$m = x^\alpha$ with $\alpha = (a_1,\ldots,a_n)
\in \mathbb{N}^n$.  The monomial $m$ is a
\emph{square-free monomial} if $a_i = 0$ or $1$ for all $i=1,\ldots,n$.
The \emph{support} of 
$m = x_1^{a_1}\cdots x_n^{a_n}$ is the set 
${\rm supp}(m) = \{j ~|~ a_j > 0 \}$.

An ideal $I \subseteq S$ is a \emph{(square-free) monomial ideal} 
if $I$  is generated by (square-free) monomials.  A monomial ideal has a unique set of minimal monomial generators denoted by $G(I)$.

Let $Q$ be a poset on the ground set $\{x_1,\ldots,x_n\}$,
where the partial order is denoted by $<_Q$.  
A poset $Q'$ is an \emph{induced poset} of $Q$
if there exists an injective function $f:Q' \rightarrow Q$
such that $x \leq_{Q'} y$ if and only if $f(x) \leq_Q f(y)$.

Associated to any poset on a finite 
ground set is a \emph{Hasse diagram}.  In particular, 
the elements of $Q$ are represented by vertices, and
there exists a line segment from $x$ to $y$ in the
``upwards'' direction if $x <_Q y$ and if there 
is no other $z \in Q$ such
that $x <_Q z <_Q y$.  The Hasse diagram is an example
of directed acyclic graph (a directed graph with no directed
cycles).  Given a poset $Q$, the number of \emph{connected
components} of $Q$, denoted $K(Q)$, is the number of connected
components of the Hasse diagram, i.e., the connected components of the Hasse diagram when viewed as a undirected graph.

An \emph{order ideal} of $Q$ is a set $A \subseteq Q$ such that if $y \in A$ and if $x <_Q y$, then $x \in A$.   
Given any monomial $m = x_1^{a_1}\cdots x_{n}^{a_n} \in S$,
we can associate with $m$ the order ideal
\[A(m) = \left\{x_j ~\left|~ \mbox{there is an}~
x_i ~\mbox{such that} ~x_j \leq_Q x_i ~~\mbox{and}~~ x_i|m\right\}\right..\]
The order ideal $A(m)$ is an induced poset of $Q$ via
the inclusion map.   We say an order ideal
$A(m)$ is \emph{connected} if the Hasse diagram
of $A(m)$ is connected. The next lemma follows directly from
the definitions.

\begin{lemma}\label{supportlem1}
Fix a poset $Q$ on $\{x_1,\ldots,x_n\}$,
and let $m_1,m_2 \in S$ be two monomials.
If ${\rm supp}(m_1) = {\rm
supp}(m_2)$, then $A(m_1) = A(m_2)$.
\end{lemma}

The next lemma will be used in
future calculations.

\begin{lemma}\label{lemmaA}
    Fix a poset $Q$ on $\{x_1,\ldots, x_n\}$ and let $m_1,m_2\in S$ be two monomials. Then
    $$A(\mathrm{lcm}(m_1,m_2))=A(m_1m_2)=A(m_1)\cup A(m_2).$$
\end{lemma}

\begin{proof}
    Note that  $$\mathrm{supp}(\mathrm{lcm}(m_1,m_2))=\mathrm{supp}(m_1m_2)=\mathrm{supp}(m_1)\cup\mathrm{supp}(m_2).$$
Now apply Lemma \ref{supportlem1}.
\end{proof}

The next lemma shows the relationships between the components of a monomial and the order subideals of its order ideal.

\begin{lemma}\label{lemmaB}
Fix a poset $Q$ on $\{x_1,\ldots, x_n\}$
and let $m \in S$ be a monomial. For any order ideal $O\subset Q$ such that $O=A(m')$ for some $m'|m$, there is a unique monomial $m_O$ satisfying:
\begin{enumerate}
\item[$\bullet$] $m_O|m$.
\item[$\bullet$] $O=A(m_O)$.
\item[$\bullet$] For any other monomial $m''|m$ such that $O = A(m'')$, we have $m''|m_O$.
\end{enumerate}
\end{lemma}

\begin{proof}
    Define
    $$m_O=\mathrm{lcm}\left\{m''\ \left|~~~ \mbox{$m''$ a monomial,} ~~m''|m,\text{ and } A(m'')=O\right\}\right..$$
    The set on the right contains $m'$ so $m_O$ is well-defined and it is clear that $m_O|m$. From the last lemma, we have $A(m_O)=O$ and from the definition, if any other monomial $m''|m$ satisfies $A(m'')=O$, then we have $m''|m_O$.
\end{proof}

\subsection{$Q$-Borel ideals}
$Q$-Borel 
ideals were introduced by 
Francisco, Mermin, and Schweig \cite{qborel}
to generalize properties of
Borel monomial ideals.
 We recall this definition.

\begin{definition}\label{defn-qborel}
     Let $I \subseteq S$ be a monomial ideal and let $Q$
    be a poset on $\{x_1,\ldots,x_n\}$.  The ideal
    $I$ is a \emph{$Q$-Borel ideal} if whenever
    $x_j \leq_Q x_i$ and $x_i|m$ for some monomial $m \in I$,
    then $x_j \cdot (m/x_i) \in I$.   We say that
    $I$ is \emph{Borel with respect to $Q$}.
\end{definition}

\begin{remark}
    Definition \ref{defn-qborel} generalizes the notion
    of a Borel monomial ideal. More precisely, a
    $Q$-Borel ideal is a \emph{Borel ideal} if
    $Q$ is the chain
    $Q=C: x_1 <_Q x_2 <_Q < \cdots <_Q x_n$.
    Note that any monomial ideal $I$ is 
    a $Q$-Borel ideal if we take $Q$
    to be the anti-chain.
\end{remark}
If $x_i|m$ and $x_j \leq_Q x_i$, then we call $x_j \cdot (m/x_i)$ a \emph{$Q$-Borel move} of the monomial $m$.
It follows that a monomial ideal $I$ is a $Q$-Borel ideal if $I$ is closed under $Q$-Borel moves.  Observe that if $m = x^\alpha$, then a $Q$-Borel move
$x_j \cdot (m/x_i)$
corresponds to the existence of a vector $e_{(i,j)} \in \mathbb{N}^n$ whose $k$-th coordinate is given by
\begin{equation}\label{defn-eij}
(e_{(i,j)})_k
=\begin{cases}
1&k=j\text{ and }x_j\leq_Q x_i\\
-1&k=i\text{ and }x_j\leq_Q x_i\\
0&\text{otherwise}\end{cases}
\end{equation}
such that $x^{\alpha+e_{(i,j)}} = x_j\cdot (m/x_i)$.  The following
lemma shall be useful.

\begin{lemma}\label{basislemma}
Fix a poset $Q$ on $\{x_1,\ldots,x_n\}$.
Suppose that $x^\alpha$ and $
x^\beta$ are monomials
of $S$ such that $x^\beta$ can be obtained
via a series of $Q$-Borel moves
on $x^\alpha$.  Then there exists 
$e_{(i_1,j_1)},\ldots,e_{(i_l,j_l)}$,
not necessarily distinct, with
$i_t \in {\rm supp}(x^\alpha)$
for $t=1,\ldots,l$, such that
\[
\alpha+e_{(i_1,j_1)}+\cdots + e_{(i_l,j_l)} = \beta.\]
Equivalently, expressed in terms of monomials, we have
$$x^{\beta}=x^\alpha\cdot\frac{x_{j_1}\cdots x_{j_l}}{x_{i_1}\cdots x_{i_l}}$$
where $x_{i_t}$ divdes $x^\alpha$ for
$t=1,\ldots,l$.
\end{lemma}

\begin{proof}
Because $x^\beta$ can be obtained
from $x^\alpha$ by $Q$-Borel moves,
there exists monomials $x^\alpha=x^{\alpha_1},x^{\alpha_2},\ldots,x^{\alpha_{r-1}},x^{\alpha_r} =x^{\beta}$ such
that $x^{\alpha_{t+1}}$ is obtained from
$x^{\alpha_{t}}$ via a $Q$-Borel move
for $t=1,\ldots,r-1$.
In particular, there
exists a vector
of the form
$e_{(a_t,b_t)}$ such that
\[\alpha_{t} +e_{(a_t,b_t)} = 
\alpha_{t+1} ~~\mbox{for each $t=1,\ldots,r-1$}\]
where $a_t \in {\rm supp}(x^{\alpha_t})$
and $x_{b_t} \leq_Q x_{a_t}$.  
Consequently,
\[\alpha + e_{(a_1,b_1)}+\cdots + e_{(a_{r-1},b_{r-1})} = \beta.\]
If $a_t \in {\rm supp}(x^\alpha)$ 
for all $t=1,\ldots,r-1$, then we are done.

On the other hand,
suppose that there is some
$e_{(a_t,b_t)}$ such that $a_t \not\in{\rm supp}(x^\alpha)$.  Let $t$ be the
smallest index such that
$a_t \not\in{\rm supp}(x^\alpha)$.  That is,
$t$ is the smallest index such that
$\alpha_{t+1}$ has not been expressed
in the form $\alpha+e_{(i_1,j_1)} + \cdots + 
e_{(i_t,j_t)}$ with all $i_k \in {\rm supp}(x^\alpha)$.
Note that $t \geq 2$ since $a_1
\in {\rm supp}(x^\alpha)$.   
Now
\[\alpha_{t+1}  = \alpha_t + e_{(a_t,b_t)} = (\alpha + e_{(a_1,b_1)}+\cdots+ e_{(a_{t-1},b_{t-1})})+ 
e_{(a_t,b_t)}.\]
Because $a_t$ is not in the support
of $x^\alpha$, but in the support
of $x^{\alpha_t}$, this means that
$a_t \in \{b_1,\ldots,b_{t-1}\}$
since the $b_k$'s correspond to the 
supports of the new
variables by which we multiply after dividing
by $a_k$.  Say $a_t = b_s$
with $s \in \{1,\ldots,t-1\}$. 
But then by equation
\eqref{defn-eij}
\[e_{(a_s,b_s)}+e_{(a_t,b_t)} = e_{(a_s,b_t)},\]
that is, the coordinate which is
$1$ in the first vector cancels
out with $-1$ in the second vector.  Furthermore,
 $b_t \leq_Q a_s$ since
$b_t \leq_Q a_t = b_s \leq _Q a_s$.
So, we can rewrite $\alpha_{t+1}$
as
\begin{eqnarray*}
\alpha_{t+1} &=&
\alpha + e_{(a_1,b_1)}+\cdots+
(e_{(a_s,b_s)}+e_{(a_t,b_t)})+ \cdots + e_{(a_{t-1},b_{t-1})}
\\
&=& \alpha +e_{(a_1,b_1)}+\cdots+ e_{(a_s,b_t)}
+ \cdots + e_{(a_{t-1},b_{t-1})}
\end{eqnarray*}
where all the $a_k$'s are in the ${\rm supp}(x^\alpha)$.
So $\alpha_{t+1}$ has the desired form.  

Repeating this
process allows 
$\beta$ to be expressed in the desired form.
\end{proof}

Because $Q$-Borel ideals are closed under $Q$-Borel
moves, the generators of $Q$-Borel ideals can be described 
as subsets of monomials of $S$ from which other monomial
generators in the ideal can be obtained via $Q$-Borel moves.  
The following terminology shall be helpful.

\begin{definition}\label{defn-qborel-gens}
Let $X$ be a subset of monomials of $S$. 
The smallest $Q$-Borel ideal $I$ that contains $X$ is 
denoted $Q(X)$, and we say $X$ is a \emph{$Q$-Borel 
generating set} of $I = Q(X)$.  
A square-free monomial ideal $J$
is a \emph{square-free $Q$-Borel ideal} if 
it is generated by the square-free monomials
of a $Q$-Borel ideal.  Given a set $Y$
of square-free monomials, we
let $sfQ(Y)$ denote the smallest square-free
$Q$-Borel ideal containing $Y$.
\end{definition}

The following fact follows directly
from the definitions.

\begin{lemma}{\cite[Proposition 2.6]{qborel}}\label{degree-gens}
If all the monomials of $X$ have the 
same degree, then all the minimal
generators of the $Q$-Borel ideal
$I =Q(X)$ have the same degree.
\end{lemma}

\subsection{$Q$-Borel principal ideals}
We are primarily interested in the 
following ideals.

\begin{definition}\label{defn-principal}
If $X = \{m\}$ contains
a single monomial, then we call $I = Q(X)$ 
a \emph{$Q$-Borel principal
ideal}, and we abuse notation and write $I =Q(m)$.
Similarly, if $Y = \{m\}$ contains a 
single square-free monomial, then we call
$I = sfQ(Y)$ a \emph{square-free $Q$-Borel principal
ideal} and write $I = sfQ(m)$.
\end{definition}

Principal $Q$-Borel ideals
are preserved under ideal multiplication.

\begin{lemma}\label{lem-principalproducts}
Fix a poset $Q$ on $\{x_1,\ldots,x_n\}$,
and let $m_1,m_2 \in S$ be two monomials.
Then
\[Q(m_1)Q(m_2) = Q(m_1m_2).\]
\end{lemma}

\begin{proof}
Let $p_1 \in Q(m_1)$, respectively
$p_2 \in Q(m_2)$, be any
monomial generator of $Q(m_1)$, respectively $Q(m_2)$.
So $p_1$ is a $Q$-Borel move of $m_1$, and
similarly for $p_2$ and $m_2$. Thus
\[p_1 = m_1\frac{x_{j_1}x_{j_2}\cdots x_{j_r}}{x_{i_1}x_{i_2}\cdots x_{i_r}} \mbox{ with $x_{j_\ell} <_Q x_{i_\ell}$ for
$\ell =1,\ldots,r$}\]
and
\[p_2 = m_2\frac{x_{b_1}x_{b_2}\cdots x_{b_s}}{x_{a_1}x_{a_2}\cdots x_{a_s}} \mbox{ with $x_{b_\ell} <_Q x_{a_\ell}$ for
$\ell =1,\ldots,s$}.\]
But this means that 
\[p_1p_2 = m_1m_2\frac{x_{j_1}x_{j_2}\cdots x_{j_r}}{x_{i_1}x_{i_2}\cdots x_{i_r}}\frac{x_{b_1}x_{b_2}\cdots x_{b_s}}{x_{a_1}x_{a_2}\cdots x_{a_s}}\]
is a $Q$-Borel move of $m_1m_2$, so $p_1p_2 \in Q(m_1m_2)$, thus showing $Q(m_1)Q(m_2) \subseteq Q(m_1m_2)$.

For the reverse containment, if $p \in Q(m_1m_2)$
is a generator of $Q(m_1m_2)$ obtained via a series
of $Q$-Borel moves on $m_1m_2$. So, by Lemma \ref{basislemma} and Lemma \ref{lemmaA}, we have
\[p = m_1m_2\frac{x_{j_1}x_{j_2}\cdots x_{j_r}}{x_{i_1}x_{i_2}\cdots x_{i_r}}\frac{x_{b_1}x_{b_2}\cdots x_{b_s}}{x_{a_1}x_{a_2}\cdots x_{a_s}}
\frac{x_{c_1}x_{c_2}\cdots x_{c_t}}{x_{d_1}x_{d_2}\cdots x_{d_t}}
\]
where $x_{j_\ell}<_Q x_{i_\ell}$, $x_{b_\ell}<_Q x_{a_\ell}$, $x_{c_\ell}<_Q x_{d_\ell}$ and ${i_\ell}\in\mathrm{supp}(m_1)\setminus\mathrm{supp}(m_2)$,  ${a_\ell}\in\mathrm{supp}(m_2)\setminus\mathrm{supp}(m_1)$ and ${d_\ell}\in\mathrm{supp}(m_1)\cap\mathrm{supp}(m_2)$ for all relevant $\ell$. Since $x_{d_1}\cdots x_{d_t}|(\mathrm{gcd}(m_1,m_2))^2$, we can re-index, if necessary, so that for some $1\leq t'\leq t-1$ we have 
$$x_{d_1}\cdots x_{d_{t'}}|{\rm gcd}(m_1,m_2)
~\mbox{and}~  x_{d_{t'+1}}\cdots x_{d_t}|\mathrm{gcd}(m_1,m_2).$$
We then have
$$\left(m_1\frac{x_{j_1}\cdots x_{j_r}}{x_{i_1}\cdots x_{i_r}}\frac{x_{c_1}\cdots x_{c_{t'}}}{x_{d_1}\cdots x_{d_{t'}}}\right)\in Q(m_1),\ \left(m_2\frac{x_{b_1}\cdots x_{b_s}}{x_{a_1}\cdots x_{a_s}}\frac{x_{c_{t'+1}}\cdots x_{c_t}}{x_{d_{t'+1}}\cdots x_{d_t}}\right)\in Q(m_2)$$

\noindent implying that $p\in Q(m_1)Q(m_2)$. Therefore, $Q(m_1m_2)\subseteq Q(m_1)Q(m_2)$ and we have the conclusion.
\end{proof}

We also have the following property of ideal intersections.

\begin{lemma}\label{intersections}
Fix a poset $Q$ on $\{x_1,\ldots,x_n\}$,
and let $m_1,m_2 \in S$ be two monomials.
If $A(m_1) \cap A(m_2) = \emptyset$, then
$Q(m_1) \cap Q(m_2) = Q(m_1)Q(m_2).$
\end{lemma}

\begin{proof}
It suffices to show that $Q(m_1) \cap Q(m_2) \subseteq
Q(m_1)Q(m_2)$. 

Note that for any monomial $m$, if $p \in G(Q(m))$ is
a minimal generator of $Q(m)$, then 
$\{x_j ~|~ j \in {\rm supp}(p)\} \subseteq A(m)$.  
In fact, we have
$$A(m) = \bigcup_{p\in G(Q(m))} \{x_j ~|~ j \in {\rm supp}(p)\}.$$
That is, $A(m)$ is precisely the set of variables
that divide at least one minimal generator of $Q(m)$.

Because $A(m_1)$ and $A(m_2)$ are disjoint, this
implies that for any $Q$-Borel movement $m'$ of $m_1$ and any $Q$-Borel movement $m''$ of $m_2$, 
$\mathrm{gcd}(m',m'')=1,$ and thus
${\rm lcm}(m',m'') =m'm''$. It then follows that
\begin{eqnarray*}
Q(m_1) \cap Q(m_2) &= &
\langle {\rm lcm}(m',m'') = m'm'' ~|~
m' \in G(Q(m_1)) ~\mbox{and}~ m'' \in G(Q(m_2)) \rangle \\
& = &
Q(m_1)Q(m_2),
\end{eqnarray*}
as desired.
\end{proof}

As first shown by Francisco, {\it et al} \cite{qborel},
the associated primes of principal $Q$-Borel
ideals are related to order ideals of $Q$.  Recall
that for any ideal $I \subseteq S$, a prime
ideal $P$  is an \emph{associated prime}
of $I$ if there exists an element $f \in S$
such that 
\[I:\langle f \rangle = \{g \in S ~|~ gf \in I\} = P.\]
We denote the set of all associated primes
of $I$ by ${\rm ass}(I)$.  We then have:

\begin{theorem}{\cite[Theorem 4.3]{qborel}}
\label{assprimechar}
Let $I = Q(m)$ for some monomial $m$ and
poset $Q$.  Then
$P \in {\rm ass}(I)$ if and only if 
\[P = \langle x_i ~|~ x_i \in A(m')\rangle\] 
for some $m'|m$  with the property that $A(m')$
is connected.
\end{theorem}

\begin{remark}
As we will see in Section 4, principal $Q$-Borel
ideals are products of prime monomial ideals, that
is, all principal $Q$-Borel ideals are examples
of ideals that are products of ideals generated by 
linear forms.  There are a number of
papers on this topic, for example \cite{CH, CT}.

In particular, the primary decomposition of principal
$Q$-Borel ideals can also be deduced 
from the work of \cite{CH}. We use the statement of
\cite{qborel} since it relates the 
associated primes directly to the 
Hasse diagram of $Q$,
\end{remark}
\begin{example}
\label{aspi}\normalfont 
We illustrate some of the above ideas with the following
example.
Let $S=\mathbb{K}[x_1,\ldots,x_{11}]$ and let $Q$ be the poset on $\{x_1,\ldots,x_{11}\}$ with Hasse diagram:
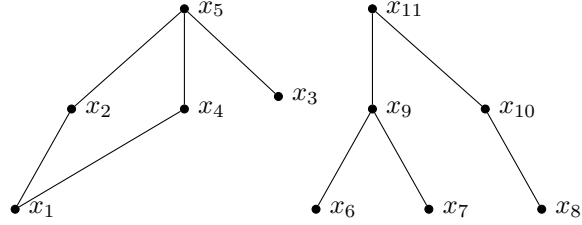
\begin{center}
\begin{tikzpicture}[scale=1, vertices/.style={draw, fill=black, circle, minimum size=3pt, inner sep=0pt}]
               \node [vertices, label=right:{${x}_{1}$}] (0) at (-3+0,0){};
               \node [vertices, label=right:{${x}_{3}$}] (4) at (-3+3+0.5,2-0.5){};
               \node [vertices, label=right:{${x}_{6}$}] (5) at (-3+3+1,0){};
               \node [vertices, label=right:{${x}_{7}$}] (7) at (-3+4.5+1,0){};
               \node [vertices, label=right:{${x}_{8}$}] (8) at (-3+6+1,0){};
               \node [vertices, label=right:{${x}_{2}$}] (1) at (-2.25+0,1.33333){};
               \node [vertices, label=right:{${x}_{4}$}] (2) at (-2.25+1.5,1.33333){};
               \node [vertices, label=right:{${x}_{9}$}] (6) at (-2.25+3+1,1.33333){};
               \node [vertices, label=right:{${x}_{10}$}] (9) at (-2.25+4.5+1,1.33333){};
               \node [vertices, label=right:{${x}_{5}$}] (3) at (-.75+0,2.66667){};
               \node [vertices, label=right:{${x}_{11}$}] (10) at (-.75+1.5+1,2.66667){};
       \foreach \to/\from in {0/1, 0/2, 1/3, 2/3, 4/3, 5/6, 6/10, 7/6, 8/9, 9/10}
       \draw [-] (\to)--(\from);
       \end{tikzpicture}
      \end{center}
      In the above drawing, $x_i <_Q x_j$ if there is a path
      from $x_i$ to $x_j$ such that the path from $x_i$ to $x_j$
      only moves ``upward".  For example
      $x_1 <_Q x_5$, but $x_1$ and $x_3$ are not comparable.
      
      If we consider the monomial $m =x_4x_9^2$, then since
      $x_1 <_Q x_4$ and $x_4|m$, the monomial $x_1\cdot (m/x_4) = x_1x_9^2$ is a $Q$-Borel move of $m$.
      The $Q$-Borel principal ideal $I=Q(x_4x_{9}^2)$
      is the monomial ideal generated by all the
      $Q$-Borel moves one can obtain from $x_4x_9^2$.  In particular,
      \small
      \[Q(x_4x_9^2) = \langle 
      x_1x_6^2,x_1x_6x_7,x_1x_7^2,
      x_1x_6x_9,x_1x_7x_9,
      x_1x_9^2,
      x_4x_6^2,x_4x_6x_7,x_4x_7^2,
      x_4x_6x_9,x_4x_7x_9,
      x_4x_9^2\rangle.\]
      \normalsize
      Observe that all the generators of $Q(x_4x_9^2)$ have degree three,
      as expected by Lemma \ref{degree-gens}.   
      
      We apply Theorem \ref{assprimechar} to compute 
      ${\rm ass}(Q(x_4x_9^2))$.  The monomials that divide 
      $x_4x_9^2$ are $x_4,x_9,x_9^2,x_4x_9$ and $x_4x_9^2$.  Now $A(x_4x_9) = A(x_4x_9^2) = \{x_1,x_4,x_6,x_7,x_9\}$ is not connected, but
      the order ideals $A(x_4) = \{x_1,x_4\}$ and
      $A(x_9) = A(x_9^2) =\{x_9,x_6,x_7\}$ are.  So
      \[{\rm ass}(Q(x_4x_9^2)) = \{\langle x_1,x_4\rangle,\langle x_6,x_7,x_9 \rangle\}.\]
      \end{example}


\section{The ideal containment problem for $Q(m)$}\label{sect.containmentproblem}

The \emph{$d$-th symbolic power} of an ideal 
$I \subseteq S$, denoted $I^{(d)}$, is the ideal
\[I^{(d)} = \bigcap_{P \in {\rm ass}(I)} (I^dS_P \cap S) \]
where $S_P$ is the ring $S$ localized at the ideal
$P$, and the intersection is over the set of all
the associated primes of $I$.  
(The definition of symbolic powers is not uniform
in the literature, where in some references, the indexing set is only over the minimal associated primes, as in \cite[Definition 4.3.22]{V}.)

The regular $d$-th power of $I$, that is $I^d$, always 
satisfies $I^d \subseteq I^{(d)}$.  
Ein-Lazersfeld-Smith
\cite{ELS} and 
Hochster-Huneke \cite{HH} showed
that, for every positive integer $d$, there is an integer $r \geq d$ such that $I^{(r)} \subseteq  I^d$. 
The ``ideal containment problem'' pertains to the 
problem of determining, for each positive
integer $d$, the smallest integer
$r$ such that $I^{(r)} \subseteq I^d$.  In this
section, we show that for any principal $Q$-Borel
ideal, we can take $r=d$.

The following
results of Cooper, Embree, H\`a, and Hoefel \cite{spmi} about symbolic powers of monomial ideals
will
be useful. If $I = Q_1 \cap \cdots \cap Q_s$ is a  
primary decomposition of the monomial ideal $I$, and
if $P \in {\rm ass}(I)$, then we define
\[Q_{\subseteq P} = \bigcap_{\sqrt{Q_i} \subseteq P} Q_i.\]
That is, $Q_{\subseteq P}$ is 
the intersection of all the primary ideals in the primary decomposition
of $I$ such that $\sqrt{Q_i}$ is contained in $P$.  Then
we have:

\begin{theorem}{\cite[Theorem 3.7]{spmi}}\label{monomialsymbolic}
The $d$-th symbolic power of a monomial ideal $I$ is 
\[I^{(d)} = \bigcap_{P \in {\rm maxass}(I)} Q_{\subseteq P}^d\]
where ${\rm maxass}(I)$ denotes the maximal associated primes of $I$, ordered by inclusion.  
 \end{theorem}

 Thus, to compute the symbolic powers of principal
 $Q$-Borel ideals, we need to determine 
 ${\rm maxass}(I)$.  We introduce the following terminology.

\begin{definition}
    Let $S=\mathbb{K}[x_1,\ldots,x_n]$ and let $Q$ be a poset over its variables. Fix a monomial $m\in S$ and
    suppose that $m'|m$.  We say that $m'$ is a \emph{maximal connected component} of $m$ if 
    \begin{enumerate}
        \item[$\bullet$] $A(m')$ is connected,
        \item[$\bullet$] $A(m')$  is maximal with respect to
    inclusion, i.e., there is no other $m''$ that 
    divides $m$ such that $A(m'')$ is connected and $A(m') \subsetneq A(m'')$, and 
        \item[$\bullet$] $m'=m_O$ with $O=A(m')$, i.e., $m'$ is the 
        unique monomial of Lemma \ref{lemmaB}.
    \end{enumerate}
    \end{definition}

Note that 
    by Lemma \ref{lemmaB}, the maximal connected components of a monomial exist and are unique. 

\begin{remark}
    Using Lemma \ref{lemmaB}, we can give
    an equivalent definition of 
    a maximal connected component in terms of the poset $Q$.  Specifically,
    let $m$ be a monomial and $Q$ a poset as before. Let $L$ be the lattice of divisors of $m$ and $\Lambda$ the subposet of $L$ consisting of $\{\mu\ |\ A(\mu)\mbox{ is connected.}\}$. Then $m'$ is a maximal connected component if and only if $m'$ is a maximal element of $\Lambda$.  This alternative viewpoint may be helpful.
\end{remark}

\begin{lemma}\label{maximalcomponents}
Let $I = Q(m)$ for some monomial $m$ and
poset $Q$.   Then $P \in {\rm maxass}(I)$ if and only if 
$P = \langle x ~|~ x \in A(m') \rangle$ with $m'$ a
maximal connected component of $m$. 
\end{lemma}

\begin{proof}
$(\Rightarrow)$ Suppose that $P \in {\rm maxass}(I)$.  
By Theorem \ref{assprimechar}, there exists a monomial $m'$ such 
that $m'|m$, $A(m')$ is connected, and $P = \langle x_i ~|~ x_i
\in A(m') \rangle$. We can assume that $m'=m_O$ with $O=A(m')$.  If $m'$ is not a maximal
connected component of $m$, then there
is some $m''$ that divides $m$ such that 
the connected component $A(m'')$ properly contains $A(m')$.
But since $A(m'')$ is connected, $P' = \langle x_i ~|~ x_i \in A(m'') \rangle$
is an associated prime of $I$ that properly contains $P$, contradicting the
maximality of $P$.  We now have the desired contradiction.

$(\Leftarrow)$ We reverse the above argument. 
Let $m'$ be a maximal connected component of $m$. By Theorem
\ref{assprimechar}, there is a prime ideal $P \in {\rm ass}(I)$
such that $P = \langle x_i ~|~ x_i \in A(m') \rangle$
since $A(m')$ is connected.  If $P$
is not a maximal associated prime, then there is a prime ideal $P'$ with $P \subsetneq P'$.  But
then $P' = \langle x_i ~|~ x_i \in A(m'') \rangle$ for some
$m''$ such that $m''|m$ and $A(m'')$ is connected.
But then $A(m') \subsetneq A(m'')$ contradicting the fact
that $m'$ is a maximal connected component of $m$.
\end{proof}

The following lemma on distinct maximal connected 
components is required.

\begin{lemma}\label{lemma.maxfactors}
    Let $m\in S$ be a monomial, and let
    $m_1$ an $m_2$  be two distinct maximal connected components of $m$.  
    Then $A(m_1) \cap A(m_2) = \emptyset$.
    \end{lemma}

\begin{proof}
Suppose that $y \in A(m_1) \cap A(m_2)$.  Then $y$ 
is path connected to every element in $A(m_1)$, and 
similarly, to every element in $A(m_2)$ since
both $A(m_1)$ and $A(m_2)$ are connected.
But then $A({\rm lcm}(m_1,m_2))$ is a connected
component of $A(m)$ that properly contains
$A(m_1)$ and $A(m_2)$.  But this contradicts the 
fact that $A(m_1)$ and $A(m_2)$ are  maximal.
\end{proof}

\begin{lemma}\label{decomposition}
    Let $m\in S$ be a monomial and let $m_1,\ldots,m_r$ be all the maximal connected components of $m$. 
    Then $m = m_1\cdots m_r$.
\end{lemma}
\begin{proof}
Note that by Lemma \ref{lemma.maxfactors}, it
follows that all the supports of $m_1,\ldots,m_r$
are pairwise disjoint, so $m_1\cdots m_r$ divides $m$.
If $m_1 \cdots m_r$ strictly divides $m$, that means that
there is either: (1) a variable $x_j$ that divides $m$ that 
does not divide any of $m_1,\ldots,m_r$, or (2)
a variable $x_j$ such that $x_j^d|m$ and $x_j^a$ divides
some $m_i$, but $a <d$.
We show that neither case can happen.

If $x_j|m$, then $A(x_j) \subseteq A(m)$ and $A(x_j)$ is connected.  Consider all
$m'$ such that $m'|m$, $A(x_j) \subseteq A(m')$,
and $A(m')$ is connected. In addition, suppose
$m'$ is picked to be maximal with the property
with respect to both inclusion and the degree of $m'$. But then $m'$ would be a maximal connected component, which is a contradiction.

For case (2), suppose that $x_j^d|m$.  Since the
$m_1,\ldots,m_r$ have distinct support, $x_j$ can only divide
one of these monomials.  After relabeling, suppose $x_j|m_1$.
Suppose $x_j^a$ with $a \geq 1$ is the largest power of $x_j$ that divides
$m_1$.  We claim that $a=d$.  Since $m_1 |m$ we know $a\leq d$. If $1 \leq a<d$, then $A(m_1x_j) = A(m_1)$ since $m_1$ and 
$m_1x_j$ have the same support. But then $m_1$ is not
a maximal connected component since $\deg m_1x_j > \deg m_1$
and $m_1x_j|m$.  So case (2) cannot happen.
\end{proof}

We relate the primary decomposition of $Q(m)$
with its maximal connected components.
\begin{lemma}\label{Lemma.maxdecomp}
Let $m\in S$ be a monomial and let $m_1,\ldots,m_r$ be all the maximal connected components of $m$. Then
$$Q(m)=Q(m_1)\cap\cdots\cap Q(m_r).$$
Furthermore, if $Q(m) = Q_1 \cap \cdots \cap Q_s$ is
a primary decomposition of $Q(m)$, then
$$Q(m_i)=Q_{\subseteq \langle A(m_i)\rangle} ~~\mbox{for
$i=1,\ldots,r$}$$
where $\langle A(m_i) \rangle = \langle x ~|~ x \in A(m_i) \rangle$.
\end{lemma}

\begin{proof}
By Lemma \ref{decomposition} we have
$m=m_1\cdots m_r$.  By Lemma \ref{lemma.maxfactors} and Lemma \ref{lemmaA},
we have that $A(m_1\cdots m_{j-1}) \cap
A(m_j) = \left(\bigcup_{i=1}^{j-1} A(m_i)\right)\cap A(m_j) = \emptyset$ for $j=2,\ldots,r$.  So by
repeatedly applying Lemma \ref{intersections}
we have
$$Q(m)=\prod_{i=1}^r Q(m_i)=\bigcap_{i=1}^r Q(m_i).$$

    For the second claim, observe that any associated prime of $Q(m)$ is an associated prime of $Q(m_j)$ for just one $j$ (due to Theorem \ref{assprimechar} and the definition of a maximal connected component); for the same reason, any associated prime of $Q(m_i)$ is an associated prime of $Q(m)$. Since $Q(m_i)$ has just one maximal associated prime, 
    namely, $\langle A(m_i) \rangle$, we 
    then have
    $Q_{\subseteq \langle A(m_i)\rangle}=Q(m_i),$
    as desired.
\end{proof}

We arrive at the main result of this section.

\begin{theorem} \label{maintheoremsymbolic}
Let $I = Q(m)$ for some monomial $m$ and
poset $Q$. Then 
$$I^{(d)}=I^d ~~\mbox{for all $d  \geq 1$}.$$
\end{theorem}

\begin{proof} Let $m_1,\ldots,m_r$ be the maximal connected components of $m$. 
By Lemma \ref{maximalcomponents}, 
${\rm maxass}(I) = \{ \langle A(m_i) \rangle ~|~
i=1,\ldots,r\}$.  
By Theorem \ref{monomialsymbolic} and Lemma \ref{Lemma.maxdecomp} we have
$$I^{(d)}=\bigcap_{i=1}^r Q^d_{\subseteq \langle A(m_i) \rangle} = \bigcap_{i=1}^r \left(Q(m_i)\right)^d.$$

But by Lemma \ref{lem-principalproducts} we have
$$\bigcap_{i=1}^r \left(Q(m_i)\right)^d
=\bigcap_{i=1}^r \left(Q(m_i^d)\right)$$
Since $A(m_i) = A(m_i^d)$, 
it follows from Lemma \ref{lemma.maxfactors}
that all 
the generators of $Q(m_i^d)$ are relatively prime with the all  generators of $Q(m_j^d)$ for any $i\neq j$.
Thus 
$$I^{(d)}=\bigcap_{i=1}^r \left(Q(m_i^d)\right)=\prod_{i=1}^r Q(m_i^d)=Q(m^d)=Q(m)^d = I^d.$$
The third and fourth equality follow from
Lemma \ref{lem-principalproducts} and
the fact that $m=m_1\cdots m_r$.
\end{proof}

Theorem \ref{maintheoremsymbolic} allows us to compute some
invariants related to the ideal containment problem.
We recall these definitions (see \cite{CHHVT} for
more on the properties of these invariants). For a homogeneous ideal $I$,  $\alpha(I)$ denotes the smallest degree of an element in a minimal set of homogeneous generators for $I$.
For a graded $R$-module $M$, $\mu(M)$ denotes its minimal number of generators. 
\begin{definition}
Let $I$ be a homogeneous ideal of $S$.
\begin{enumerate}
\item (see \cite{BH}) The $\emph{Waldschmidt constant}$ of $I$,
denoted by $\widehat{\alpha}(I)$, is
$$\hat\alpha(I) := \lim_{s \rightarrow \infty} \frac{\alpha(I^{(s)})}{s}.$$
\item (see \cite{GGSVT}) The \emph{$d$-th symbolic defect} of $I$, denoted by $\mathrm{sdefect}(I,d)$, as
    $$\mathrm{sdefect}(I,d)=\mu\left(I^{(d)}/I^d\right).$$
\item (see \cite{BH}) The \emph{resurgence} of $I$, denoted by $\rho(I)$, is
    $$\rho(I)=\sup\left\{\frac{s}{r}\ |\ I^{(s)}\not\subset I^r\right\}.$$
\end{enumerate}
\end{definition}

\begin{corollary}\label{symboliccor}
Let $I = Q(m)$ for some monomial $m$ and
poset $Q$.  Then
    \begin{enumerate}
    \item $\widehat{\alpha}(I) = \deg(m)$,
    \item ${\rm sdefect}(I,d) = 0$ for all $d \geq 1$, and
    \item $\rho(I) = 1$.
    \end{enumerate}
\end{corollary}

\begin{proof}
These results follow directly from the fact that $I^d = I^{(d)}$ for all $d \geq 1$.
\end{proof}

\begin{remark}
    Observe that Corollary \ref{symboliccor} holds for principal ideals in the regular sense, thus
    illustrating the theme that principal $Q$-Borel ideals behave like principal ideals.
\end{remark}

\begin{remark}
For principal $Q$-Borel ideals
$I = Q(m)$, Corollary \ref{symboliccor} shows that the Waldschmidt constant is very easy to obtain from $m$.  If we consider
square-free $Q$-Borel ideals, it
becomes much harder to determine
this invariant.  In a follow up
paper \cite{CKSVT21}, we look at
the Waldschmidt constant
of square-free $Q$-Borel ideals
in the special case that $Q$ is the 
chain $C:x_1 < \cdots <x_n$, or in other words, square-free
Borel ideals.
\end{remark}


\section{Associated primes
of powers of principal $Q$-Borel ideals}

As noted in the introduction, 
studying the set of the associated
primes of a power of an ideal has been
of recent interest.   One property
that has been studied is the
persistence property.
Formally, an ideal $I$ is said
to have the \emph{persistence property} if
${\rm ass}(I^i) \subseteq {\rm ass}(I^{i+1})$ for all $i \geq 1$.
Given this interest,
it makes sense to determine
if principal $Q$-Borel ideals have this
property.  This short section
gives two different proofs that 
principal $Q$-Borel ideals
have this property. 

Our first proof relies on the work of Herzog, Rauf, and
Vladoiu \cite{HRV}; we recall a key definition from
\cite{HRV}.

\begin{definition}
A monomial ideal $I$ is 
a \emph{transversal polymatroidal} ideal
if 
\[I = P_1P_2\cdots P_t\]
for prime monomial ideals $P_1,\ldots,P_t$.
\end{definition}

\begin{lemma}\label{transversal}
Let $I=Q(m)$ for some monomial $m$ and poset $Q$.  Then $I$ is a transversal
polymatroidal ideal.
\end{lemma}

\begin{proof}
This result follows from 
\cite[Proposition 2.7]{qborel} which
states that a principal $Q$-Borel
ideal is a product of prime monomial
ideals.
\end{proof}

We then have following result,
which implies that principal
$Q$-Borel ideals have
the persistence property.  Our first
proof makes use of a property
of polymatroidal ideals,
while our second proof uses Lemma \ref{lem-principalproducts}, and is self-contained.

\begin{theorem}\label{persistence}
Let $I=Q(m)$ for some monomial $m$ and poset $Q$. Then
    we have
    \[{\rm ass}(I) = {\rm ass}(I^s) ~~
    \mbox{for all $s \geq 1$.}\]
\end{theorem}

\begin{proof}[First Proof]
By \cite[Corollary 3.6]{HRV}, 
every transversal polymatroidal
ideal $J$ satisfies ${\rm ass}(J)
= {\rm ass}(J^s)$ for all $s \geq 1$.
Now apply Lemma \ref{transversal}.
\end{proof}

\begin{proof}[Second Proof]
By repeatedly applying Lemma \ref{lem-principalproducts}, $I^s = Q(m)^s = Q(m^s)$.
If $P \in {\rm ass}(I)$, then by 
Theorem \ref{assprimechar}, there is a $m'$ such
$m'|m$ and $A(m')$ is connected
and $P = \langle x_i ~|~ x_i \in A(m') \rangle$. But then $m'|m^s$ and $A(m')$ is connected, so
$P$ is also an associated prime of $I^s = Q(m^s)$.

Conversely, suppose that $P \in 
{\rm ass}(I^s) = {\rm ass}(Q(m^s))$.  By
Theorem \ref{assprimechar}, there is a monomial 
$m'$ that divides $m^s$ such that $A(m')$ is 
connected and $P = \langle x_i ~|~ x_i \in A(m')
\rangle$.  If 
$m' = x_{i_1}^{b_{i_1}}\cdots x_{i_r}^{b_{i_r}}$
with $b_{i_j} >0$,
let $m'' =x_{i_1}\cdots x_{i_r}$.  
Since $m'|m^s$, we have $m''|m$.  Furthermore,
because $m'$ and $m''$ share the same support,
$A(m') = A(m'')$ by Lemma \ref{supportlem1}.
So, we have $m''$ divides $m$ and
$A(m'')$ is connected.  So by Theorem \ref{assprimechar}, $P = \langle x_i ~|~ x_i 
\in A(m') = A(m'') \rangle$ is an associated
prime of $I$, as desired.
\end{proof}



\section{The analytic spread of principal $Q$-Borel ideals}

In this section, we compute the analytic spread of 
principal $Q$-Borel ideals $Q(m)$
and square-free principal $Q$-Borel ideals
$sfQ(m)$.  In particular,
this invariant is expressed in terms of
the properties of the order ideal $A(m)$
viewed as an induced subposet of $Q$.  
We  recall the definition of analytic spread.

\begin{definition}\label{defn:analytic}
Let $I \subseteq S =\mathbb{K}[x_1,\ldots,x_n]$ be a homogeneous
ideal, and let ${\bf m} = \langle x_1,\ldots,x_n \rangle$.  
The \emph{analytic spread} of $I$, denoted $\ell(I)$, is the 
Krull dimension of the ring 
\[\mathcal{F}(I) = \bigoplus_{i \geq 0} \frac{I^i}{{\bf m}I^i}
~~~~\mbox{where $I^0 = S$}.\]
\end{definition}

\begin{remark}
The ring $\mathcal{F}(I)$  is usually
referred to as the \emph{special fiber ring}.   The special fiber
ring is also isomorphic to $\mathcal{R}(I)/{\bf m}\mathcal{R}(I)$ where
$\mathcal{R}(I) = R[It] = \bigoplus_{i \geq 0} I^it^i \subseteq R[t]$
is the Rees algebra of $I$.  Roughly speaking, the analytic
spread is the minimum number of generators of an ideal
$J$ that is a reduction of $I$ (e.g., see \cite[Corollary 8.2.5]{HS}).
\end{remark}

The next lemma gives us a tool to compute
$\ell(I)$ when $I$ is generated
by monomials all of the same degree.

\begin{lemma}{\cite[Lemma 3.2]{appei}}\label{ranklemma}
    Let $I= \langle x^{\alpha_1},\ldots,x^{\alpha_r} \rangle $ be a monomial ideal and let $A$ be the matrix with columns $\alpha_i$. If $\deg x^{\alpha_i}=d$ for all $i$, then
     the analytic spread of $I$ is
    $$\ell(I)=\mathrm{rank}\ A.$$
\end{lemma}

Since $I = Q(m)$ is generated by
monomials of the same degree
(see Lemma \ref{degree-gens}), to
compute $\ell(Q(m))$ it is enough
compute the rank of the matrix corresponding
to the degrees of the generators.
The rank of this matrix is 
encoded in $A(m)$, as we now show.

\begin{theorem}\label{maintheorem1}
Let $I = Q(m)$ for some monomial $m$ and
poset $Q$.
Then
$$\ell(I)=|A(m)|-K(A(m))+1$$
where $A(m)$ is the order ideal of $m$
and $K(A(m))$ is the number of connected
components of $A(m)$ as an induced 
subposet of $Q$.
\end{theorem}

\begin{proof} 
We can write $I = Q(m)$  
as $I = \langle x^{\alpha_1},
\ldots,x^{\alpha_r} \rangle$ where
$\{x^{\alpha_1},\ldots,x^{\alpha_r}\}$
are the minimal generators,
and $m = x^{\alpha_r}$.
By
Lemma \ref{degree-gens}, the generators all
have the same degree.

Let $A = \begin{bmatrix} \alpha_1 &
\cdots & \alpha_r \end{bmatrix}$ be
the $n \times r$ matrix where the
$i$-th column is given by $\alpha_i$.  
By Lemma \ref{ranklemma} we need
to compute ${\rm rank}(A)$, or equivalently,
the rank of the matrix
\[A' = \begin{bmatrix}
\alpha_1-\alpha_r & \alpha_2-\alpha_r &
\cdots & \alpha_{r-1}-\alpha_r & 
\alpha_r \end{bmatrix}\]
because the column space of 
$A$ and $A'$ is the same.

For all $x_j \leq_Q x_i$, let
$e_{(i,j)} \in \mathbb{N}^n$ denote the vector
defined in
\eqref{defn-eij}.  Note that $x^{\alpha_k}$ is the monomial
obtained from $m = x^{\alpha_r}$ via a series of $Q$-Borel moves.
In particular by Lemma \ref{basislemma} there exists vectors $e_{(i_1,j_1)},e_{(i_2,j_2)},\ldots,e_{(i_l,j_l)}$
with $i_t \in {\rm supp}(x^{\alpha_r})$ for $t=1,\ldots,l$ such that 
$$\alpha_r+e_{(i_1,j_1)} + \cdots + e_{(i_l,j_l)} = \alpha_k.$$
Thus $\alpha_k-\alpha_r\in\mathrm{Span}\{e_{(i,j)}\ |\ i\in\mathrm{supp}(x^{\alpha_r}) ~~\mbox{and}~~ x_j \leq_Q x_i\}$ for any $1\leq k\leq r-1$.  Because $\alpha_r+e_{(i,j)}$ is a column of $A$ for any $i\in\mathrm{supp}(x^{\alpha_r})$ and $x_j \leq_Q x_i$, the vectors
$e_{(i,j)}$ appear as columns of $A'$.
This implies that 
$$\mathrm{rank}\ A'=1+\dim_{\mathbb{K}} (\mathrm{Span}\{e_{(i,j)}\ |\ i\in\mathrm{supp}(x^{\alpha_r}) ~~\mbox{and}~~ x_j \leq_Q x_i\}).$$
where the $1$ corresponds to the column corresponding to $\alpha_r$.

Consider the order ideal $A(x^{\alpha_r})$ and view it as an induced poset of of $Q$.
Let $B$ denote the incidence matrix of the Hasse diagram associated to
$A(x^{\alpha_r})$.  That is, $B$ is the matrix 
whose rows are indexed
by the elements of $A(x^{\alpha_r})$ and whose columns
are indexed by the directed edges in $A(x^{\alpha_r})$.
Furthermore,
in the column indexed by the  edge of $A(x^{\alpha_r})$ between
$x_j$ and $x_i$ with
$x_j <_Q x_i$, we put a $-1$ in the row
indexed by $x_i$ and a $1$ in the row indexed
by $j$.

It follows from the proof of \cite[Theorem 8.3.1]{GR} that the kernel of $B$ is generated by the vectors $v_C=\sum_{x_i\in C} e_i$ where $C$ is a connected component of $A(m)$. Given that the columns of $B$ belong to $\mathrm{Span}\{e_{(i,j)}\ |\ i\in\mathrm{supp}(x^{\alpha_r}) ~~\mbox{and}~~ x_j \leq_Q x_i\}$ and the generators of this space are orthogonal to the elements in $\{v_C\ |\ C\mbox{ is a connected component of }A(m)\}$, then
$$\mathrm{Col}(B)=\mathrm{Span}\{e_{(i,j)}\ |\ i\in\mathrm{supp}(x^{\alpha_r}) ~~\mbox{and}~~ x_j \leq_Q x_i\}.$$

Thus, $$\ell(Q(m))=\mathrm{rank}\ A=\mathrm{rank}\ A'=1+\mathrm{rank}\ B$$ and from \cite[Theorem 8.3.1]{GR} we know $\mathrm{rank}\ B=|A(m)|-K(A(m))$ from where we obtain the desired conclusion.
\end{proof}

Before considering square-free 
principal $Q$-Borel ideals, we make a brief
aside to differentiate our work from that
of Herzog and Qureshi \cite{herzogas}.
As shown in \cite{herzogas}, 
the analytic spread of a polymatroidal ideal
(\cite[Definition 2.3]{herzogas}) can
be computed via the linear relation graph of the ideal.

\begin{definition}
 Let $G(I)=\{m_1,\ldots,m_s\}$ be the minimal
 generators of a monomial ideal $I$. 
 The \emph{linear relation graph} $\Gamma$ of $I$ is the graph with edge set
 $$E=\{\{i,j\}\ |\ \text{there exists}\ m_k,m_l\in G(I)\text{ such }x_im_k=x_jm_l\}$$
 and vertex set $V=\bigcup_{\{i,j\}\in E}\{i,j\}$.
\end{definition}

The analytic spread of a polymatroidal ideal is related to its linear
relation graph.

\begin{lemma}{\cite[Lemma 4.2]{herzogas}}\label{herzoglemma}
    Let $I$ be a polymatroidal ideal 
    with linear relation graph $\Gamma$. 
    If $r$ is the number of vertices of $\Gamma$ 
    and $s$ is the number of connected components of $\Gamma$, then
    $$\ell(I)=r-s+1.$$
\end{lemma}

As shown in \cite[Proposition 2.9]{qborel},
a principal $Q$-Borel ideal $I=Q(m)$ is a polymatroidal
ideal.  Consequently, one can compute
$\ell(Q(m))$ via Lemma \ref{herzoglemma}.
However, our Theorem \ref{maintheorem1} has
the advantage of expressing the 
analytic spread in terms of the poset $Q$
and order ideal $A(m)$.   As the next 
example shows, we do not necessarily have
$|A(m)|=r$ and $K(A(m))=s$, with $r$ and $s$ as 
in Lemma \ref{herzoglemma}.

\begin{example}
    Consider $S=\mathbb{K}[x_1,x_2,x_3]$, and let
    our poset $Q$ on $\{x_1,x_2,x_3\}$
    have Hasse diagram
    
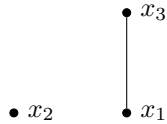
\begin{center}
\begin{tikzpicture}[scale=1, vertices/.style={draw, fill=black, circle, minimum size=3pt, inner sep=0pt}]

               \node [vertices, label=right:{${x}_{2}$}] (1) at (-2.25+0,1.33333){};
               \node [vertices, label=right:{${x}_{1}$}] (2) at (-2.25+1.5,1.33333){};
               \node [vertices, label=right:{${x}_{3}$}] (3) at (-.75+0,2.66667){};
     \foreach \to/\from in {3/2}
       \draw [-] (\to)--(\from);
       \end{tikzpicture}
      \end{center}
    Consider $I=Q(x_2x_3) = \langle x_1x_2,x_2x_3 \rangle$.  Then $|A(x_2x_3)| = 3$ and 
    $K(A(x_2x_3)) =2$.  However, the linear
    relation graph $\Gamma$ of $I$ contains
    the single edge $\{1,3\}$ since 
    $x_3(x_1x_2) = x_1(x_2x_3)$ is the only
    linear relation among the generators of $I$.  
    So $r = 2$ and $s=1$.
   \end{example}
 
 In light of the above example, 
 it is natural to ask if there is any connection between
 $A(m)$ and the linear relation graph 
 $\Gamma$ of the principal $Q$-Borel ideal $I=Q(m)$. 
 This relationship is explained in the following theorem.

\begin{theorem}\label{transclosure} 
    Fix a poset $Q$ on $X=\{x_1,\ldots,x_n\}$ and take $m\in S$ a monomial. Let $I=Q(m)$ and let $\Gamma$ be its linear relation graph. Consider $H$, the Hasse diagram of $A(m)$, but as an undirected graph; that is, the
    vertex set is $V(H)=A(m)$ and $\{x_i,x_j\}\in E(H)$ is an
    edge if $x_i<_Q x_j$ or $x_j<_Q x_i$ and there is no element $y\in X$ with $x_i <_Q y <_Q x_j$ or
    $x_j <_Q y < x_i$.
    
    Then $\Gamma$ is the transitive closure of $H$ after removing the isolated vertices of $H$.
\end{theorem}

\begin{proof}
    First, it is clear that $E(H)\subseteq E(\Gamma)$.  Also $V(H)\setminus V(\Gamma)$ is precisely the set of isolated vertices of $H$. Now, take $\{i,j\}\in E(\Gamma)$. Then there exists $x^{\alpha},x^{\beta}\in G(I)$ such that
    $$e_i+\alpha=e_j+\beta.$$
    
    But $x^\alpha,x^\beta$ are also $Q$-Borel movements of $m =x^\nu$. Then, by Lemma \ref{basislemma} there exists $i_1,\ldots,i_t\in\mathrm{supp}(m)$, $j_1,\ldots, j_t$ with $x_{j_k}<_Q x_{i_k}$ for $1\leq k\leq t$ and $\{i'_1,\ldots,i'_s\}\in\mathrm{supp}(m)$, $j'_1,\ldots, j'_s$ with $x_{j'_k}<_Q x_{i'_k}$, $1\leq k\leq s$, such that:
    $$\nu=\alpha+\sum_{k=1}^t e_{(i_k,j_k)}=\beta+\sum_{k=1}^s e_{(i'_k,j'_k)}$$
    and then
    $$e_j-e_i=\sum_{k=1}^s e_{(i'_k,j'_k)}-\sum_{k=1}^t e_{(i_k,j_k)}.$$
    
    But this means that there is a path from $i$ to $j$ along the vertices of $H$ and then $\{i,j\}$ is in the transitive closure of $H$.
\end{proof}

\begin{remark}
    The previous theorem implies that if $c$ is the
    number of isolated vertices of $A(m)$, then
    $|A(m)| = r+c$ and $K(A(m))=s+c$ where 
    $r$ and $s$ are as in Lemma \ref{herzoglemma}.  Using the
    fact that a principal $Q$-Borel ideal is a polymatroidal
    ideal, we could then use Lemma \ref{herzoglemma} and 
    Theorem \ref{transclosure} to give a different
    proof of Theorem \ref{maintheorem1}.  In particular,
    if $\Gamma$ is the linear relation graph of
    $Q(m)$, we have
    $$\ell(Q(m)) = r-s+1 = (r+c)-(s+c)+1 = |A(m)|-K(A(m))+1.$$
    Our proof of Theorem \ref{maintheorem1} avoids using
    the polymatroidal property.
\end{remark}

Our analysis of the square-free principal
$Q$-Borel case is similar to the principal
$Q$-Borel case.
We require the following notation.
Suppose that $Q$ is a poset on $X = \{x_1,\ldots,x_n\}$.  
If  $Y = \{x_{j_1},\ldots,x_{j_s}\}$ is a subset of
$X$, then $Q$ induces a poset $Q'$
on $Y$ if we define $x_j <_{Q'} x_i$ if $x_j <_Q  x_i$.  

If $m$ is a monomial only in the
variables of $Y$, then we write $A_{Q}(m)$ or $A_{Q'}(m)$ if
wish to view the order ideal in $Q$ on the set $X$ or in $Q'$
on the set $Y$.  Similarly, we write $sfQ(m)$ or $sfQ'(m)$, and $Q(m)$ and $Q'(m)$ if we wish to denote which partial order and ground set we are using.

\begin{theorem}\label{sqfanalyticspread}  
    Fix a poset $Q$ on $X = \{x_1,\ldots,x_n\}$ and suppose that  $m \in S$ is a square-free monomial. Let $m' = {\rm gcd}(G(sfQ(m)))$ be the greatest
    common divisor of all the generators of the 
    square-free principal $Q$-Borel ideal $I = sfQ(m)$.
    Then 
    \[\ell(I) = \ell(Q'(m/m')) = |A_{Q'}(m/m')| - K(A_{Q'}(m/m')) +1, \]
    where $Q'$ is the induced poset on $Y= X \setminus \{x_j ~|~ j \in {\rm supp}(m')\}$.
 \end{theorem}

\begin{proof}
Let $m = x^\alpha=x_{i_1}x_{i_2}\cdots x_{i_s}$ and
$m' = x^\delta = x_{j_1}\cdots x_{j_t}$.  Since $m'$ is
the greatest common divisor of all the generators, $m'|m$.  Furthermore,
suppose $x_j|m'$, and thus $x_j|m$.  If $x_k <_Q x_j$, then
$\frac{x_k}{x_j}m \not\in I$ because otherwise we would have a generator
of $I$ not divisible by $x_j$.  If $x_j <_Q x_i$ and $x_i|m$,
then $\frac{x_j}{x_i}m$ is a $Q$-Borel move of $m$, but
it is not in $I$ since this monomial is not square-free.  Thus, $x\in A(m')$ implies that $A(x)=\{x\}$ or for any $y\in Q\setminus\{x\}$ comparable to $x$, the corresponding $Q$-Borel movement is not in $sfQ(m)$.


We first consider the case that $m'=1$.  Note that
this means that every $x_i$ that divides $m$ is not a minimal
element of $A(m)$.  Indeed, if $x_i$ is a minimal element, then
$x_i$ would appear in every generator of $I$, contradicting the
fact $m'=1$.  

Set $I = sfQ(m)$ and $J = Q(m)$. Let $A$ be the matrix whose column entries have the
form $\beta$ where $x^\beta$ is a generator of $I$,
and similarly, let $B$ be the matrix whose columns
have the form $\gamma$ where $x^\gamma$ is
a generator of $J$.  By Lemma \ref{ranklemma}
and Theorem \ref{maintheorem1} we have
\[\ell(I) = {\rm rank}(A) \leq {\rm rank}(B) =
\ell(J) = |A(m)|- K(A(m))+1.\]
The inequality follows from the fact that all of 
the columns of $A$ are in $B$.

Fix any $i \in {\rm supp}(m)$ (and thus, $x_i|m$), and 
suppose $x_j <_Q x_i$.
If $x_j \nmid m$, then $\frac{x_j}{x_i}m\in G(I)$, and therefore $\alpha+e_{(i,j)}$ is a column of $A$.
Since $\alpha$ is a column of $A$, we have
$e_{(i,j)} \in {\rm Col}(A)$.   If $x_j$ also 
divides $m$, then there exists a minimal element 
$x_k<_Q x_j$.  Since $x_k$ is minimal, our
hypotheses imply that $x_k \nmid m$. But then
$$\alpha+e_{(i,k)}~~\mbox{and}~~ \alpha+e_{(j,k)}$$
are columns of $A$, and then $e_{(i,k)}-e_{(j,k)}=e_{(i,j)}$ is in $\mathrm{Col}(A)$. But from Theorem \ref{maintheorem1} we have
$$\mathrm{Col}(B)=\mathrm{Span}(\{\alpha\}\cup\{e_{(i,j)}\ |\ i\in\mathrm{supp}(m), x_j \leq_Q x_i\})\subset\mathrm{Col}(A).$$
Consequently, ${\rm rank}(B) \leq {\rm rank}(A)$, giving
the desired result.

Now suppose that $m'=x^\delta = x_{j_1}\cdots x_{j_k}  > 1$.
Since every generator of $sfQ(m)$ is divisible by 
$m'$, we have $$I=sfQ'(m)=m'\cdot sfQ'(m/m').$$
If $A$, respectively $B$, is the matrix whose 
columns have the form
$\gamma$ with $x^\gamma$ a generator of $sfQ'(m)$,
respectively, $sfQ'(m/m')$, we can use Lemma \ref{ranklemma}
and the proof of Theorem \ref{maintheorem1} to show that
${\rm rank}(A) = {\rm rank}(B)$;  in particular,
one needs to verify
\begin{eqnarray*}
\dim\ \mathrm{Col}(A)
  &=&\dim\mathrm{Span}(\{\delta+\beta\}\cup\{e_{(i,j)}\in \mathbb{N}^{|Q|}\ |\ i\in\mathrm{supp}(m/m')~~\mbox{and}~~ x_j\leq_{Q'} x_i\}) \\
         &=&\dim\mathrm{Span}(\{\beta\}\cup\{e_{(i,j)} 
         \in \mathbb{N}^{|Q'|}\ |\ i\in\mathrm{supp}(m/m')~~\mbox{and}~~ x_j\leq_{Q'} x_i\})\\
        &=&\dim\ \mathrm{Col}(B)
        \end{eqnarray*}
where $m = x^{\delta+\beta}$ and $m/m' = x^\beta$.
Consequently
 $$\ell(I) = {\rm rank}(A) = {\rm rank}(B) = \ell(sfQ'(m/m')) = \ell(Q'(m/m'))$$
where the last equality follows from
the first part of the proof.
\end{proof}


\begin{example}
We illustrate the above result.
Let $Q$ be the poset with Hasse diagram
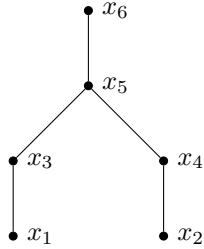
\begin{center}
\begin{tikzpicture}[scale=1, vertices/.style={draw, fill=black, circle, minimum size=3pt, inner sep=0pt}]
               \node [vertices, label=right:{${x}_{1}$}] (1) at (-1,-2){};
               \node [vertices,label=right:$x_3$] (3) at (-1,-1){};
               \node [vertices, label=right:{${x}_{2}$}] (2) at (1,-2){};
               \node [vertices, label=right:{${x}_{4}$}] (4) at (1,-1){};
               \node [vertices, label=right:{${x}_{5}$}] (5) at (0,0){};
               \node [vertices, label=right:{${x}_{6}$}] (6) at (0,1){};
     \foreach \to/\from in {1/3, 2/4, 3/5, 4/5, 5/6}
       \draw [-] (\to)--(\from);
       \end{tikzpicture}
      \end{center}
Let $m=x_1x_2x_3x_6$ and $I=sfQ(m)$, and thus
 $$I=\langle x_1x_2x_3x_6,x_1x_2x_3x_5,x_1x_2x_3x_4\rangle.$$
We have $\mathrm{gcd}(G(I))=x_1x_2x_3$.  Therefore $Q'$ 
is the poset on $\{x_4,x_5,x_6\}$ with Hasse diagram
      
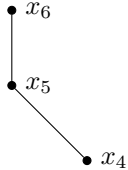
\begin{center}
\begin{tikzpicture}[scale=1, vertices/.style={draw, fill=black, circle, minimum size=3pt, inner sep=0pt}]
               \node [vertices, label=right:{${x}_{4}$}] (4) at (1,-1){};
               \node [vertices, label=right:{${x}_{5}$}] (5) at (0,0){};
               \node [vertices, label=right:{${x}_{6}$}] (6) at (0,1){};
     \foreach \to/\from in { 4/5, 5/6}
       \draw [-] (\to)--(\from);
       \end{tikzpicture}
      \end{center}
\noindent
Hence $\ell(I)=\ell(Q'(x_6))=3$.
\end{example}

\begin{remark}
It can be shown that $m' = {\rm gcd}(G(I))$ in
Theorem \ref{sqfanalyticspread} is the largest monomial (by 
degree) that divides $m$ such that 
\[\{x_j ~|~ j \in {\rm supp}(m')\} = A(m').\]  That
is, the variables that divide $m'$  form an order ideal.
Returning to the above example, note that
$\{x_j ~|~ j \in {\rm supp}(x_1x_2x_3)\}  =
\{x_1,x_2,x_3\} = A(x_1x_2x_3)$ in the poset
$Q$.  Note that if no such monomial
exists, we use the convention that
$A(1) = \emptyset$.

\end{remark}

Using the above interpretation of $m'$,
we have the following corollary, which
uses the following terminology.
Given a poset $Q$ on $\{x_1,\ldots,x_n\}$, the \emph{minimal elements}
of $\{x_1,\ldots,x_n\}$ are those $x_i$ that are minimal
with respect to the partial order on $Q$.

\begin{corollary}
    Fix a poset $Q$ and suppose that  $m \in S$ is a square-free monomial.  
    Suppose $\{x_j ~|~ j \in {\rm supp}(m)\}$
    contains no minimal elements of $Q$.
    If $I =sfQ(m)$, then 
    \[\ell(I) = \ell(Q(m)) = |A(m)| - 
    K(A(m)) +1. \]
\end{corollary}

\begin{proof}
No subset of $\{x_j ~|~ j \in {\rm supp}(m)\}$
is an order ideal in $Q$.  So $m'=
{\rm gcd}(G(I)) = 1$.  Now apply
Theorem \ref{sqfanalyticspread}.
\end{proof}

\bibliographystyle{abbrv}
\bibliography{references}

\begin{thebibliography}{10}

\bibitem{BPhd}
A.~Bhat.
\newblock {\em Associated primes and Betti splitting of some generalized Borel
  ideals}.
\newblock PhD thesis, Oklahoma State University, 2019.
\newblock Ph.D. Thesis.

\bibitem{WalSQF}
C.~Bocci, S.~Cooper, E.~Guardo, B.~Harbourne, M.~Janssen, U.~Nagel,
  A.~Seceleanu, A.~Van~Tuyl, and T.~Vu.
\newblock The {W}aldschmidt constant for squarefree monomial ideals.
\newblock {\em Journal of Algebraic Combinatorics}, 44(4):875--904, 2016.

\bibitem{BH}
C.~Bocci and B.~Harbourne.
\newblock Comparing powers and symbolic powers of ideals.
\newblock {\em Journal of Algebraic Geometry}, 19(3):399--417, 2010.

\bibitem{CKSVT21}
E.~Camps~Moreno, C.~Kohne, E.~Sarmiento, and A.~Van~Tuyl.
\newblock On the {W}aldschmidt constant of square-free principal {B}orel
  ideals.
\newblock {\em arXiv preprint arXiv:2105.07307}, 2021.

\bibitem{CHHVT}
E.~Carlini, H.~T. H{\`a}, B.~Harbourne, and A.~Van~Tuyl.
\newblock Ideals of powers and powers of ideals: Intersecting algebra,
  geometry, and combinatorics.
\newblock {\em Lecture Notes of the Unione Matematica Italiana}, 27, 2020.

\bibitem{CH}
A.~Conca and J.~Herzog.
\newblock Castelnuovo-mumford regularity of products of ideals.
\newblock {\em ollectanea Mathematica}, 54(2):137--152, 2003.

\bibitem{CT}
A.~Conca and M.~C. Tsakiris.
\newblock Resolution of ideals associated to subspace arrangements.
\newblock {\em arXiv preprint arXiv:1910.01955}, 2019.

\bibitem{spmi}
S.~M. Cooper, R.~J. Embree, H.~T. H{\`a}, and A.~H. Hoefel.
\newblock Symbolic powers of monomial ideals.
\newblock {\em Proceedings of the Edinburgh Mathematical Society},
  60(1):39--55, 2017.

\bibitem{DDGHN}
H.~Dao, A.~De~Stefani, E.~Grifo, C.~Huneke, and L.~N{\'u}{\~n}ez-Betancourt.
\newblock Symbolic powers of ideals.
\newblock In {\em Singularities and foliations. geometry, topology and
  applications}, pages 387--432. Springer, 2015.

\bibitem{ELS}
L.~Ein, R.~Lazarsfeld, and K.~Smith.
\newblock Uniform behavior of symbolic powers of ideals.
\newblock {\em Invent. Math}, 144(2):241--252, 2001.

\bibitem{fms}
C.~A. Francisco, J.~Mermin, and J.~Schweig.
\newblock Borel generators.
\newblock {\em Journal of Algebra}, 332(1):522--542, 2011.

\bibitem{qborel}
C.~A. Francisco, J.~Mermin, and J.~Schweig.
\newblock Generalizing the borel property.
\newblock {\em Journal of the London Mathematical Society}, 87(3):724--740,
  2013.

\bibitem{GGSVT}
F.~Galetto, A.~V. Geramita, Y.-S. Shin, and A.~Van~Tuyl.
\newblock The symbolic defect of an ideal.
\newblock {\em Journal of Pure and Applied Algebra}, 223(6):2709--2731, 2019.

\bibitem{GR}
C.~Godsil and G.~F. Royle.
\newblock {\em Algebraic graph theory}, volume 207.
\newblock Springer Science \& Business Media, 2013.

\bibitem{H}
J.~Herzog et~al.
\newblock Generic initial ideals and graded betti numbers.
\newblock In {\em Computational commutative algebra and combinatorics}, pages
  75--120. Mathematical Society of Japan, 2002.

\bibitem{herzogas}
J.~Herzog and A.~A. Qureshi.
\newblock Persistence and stability properties of powers of ideals.
\newblock {\em Journal of Pure and Applied Algebra}, 219(3):530--542, 2015.

\bibitem{HRV}
J.~Herzog, A.~Rauf, and M.~Vladoiu.
\newblock The stable set of associated prime ideals of a polymatroidal ideal.
\newblock {\em Journal of Algebraic Combinatorics}, 37(2):289--312, 2013.

\bibitem{HH}
M.~Hochster and C.~Huneke.
\newblock Comparison of symbolic and ordinary powers of ideals.
\newblock {\em Inventiones mathematicae}, 147(2):349--369, 2002.

\bibitem{HS}
C.~Huneke and I.~Swanson.
\newblock {\em Integral closure of ideals, rings, and modules}, volume~13.
\newblock Cambridge University Press, 2006.

\bibitem{appei}
J.~Mart{\'\i}nez-Bernal, S.~Morey, and R.~H. Villarreal.
\newblock Associated primes of powers of edge ideals.
\newblock {\em Collectanea Mathematica}, 63(3):361--374, 2012.

\bibitem{V}
R.~Villarreal.
\newblock {\em Monomial algebras, 2nd edition}.
\newblock CRC Press, 2015.

\end{thebibliography}

\end{document}